\documentclass[11pt]{amsart}
\usepackage{amsfonts,amssymb,amsthm}
\usepackage{amsmath,amscd}
\usepackage{pstricks}
\usepackage{pstricks,pst-node}
\usepackage{mathrsfs}
\usepackage[all]{xy}

\DeclareMathAlphabet{\mathpzc}{OT1}{pzc}{m}{it}

\renewcommand{\subsection}[1]{\vspace{.18in}
\par\noindent\addtocounter{subsection}{1}
\setcounter{equation}{0}{\bf\thesubsection.\hspace{5pt}#1}}

\theoremstyle{definition}

\theoremstyle{plain}

\newtheorem{Prop}[subsection]{Proposition}
\newtheorem{Thm}[subsection]{Theorem}

\newtheorem{Lem}[subsection]{Lemma}
\newtheorem{Coro}[subsection]{Corollary}
\numberwithin{equation}{subsection}

\newcommand{\Vnk}{\sV_\field(n)}
\newcommand{\Vnkh}{\sV_\field(n)_h}
\newcommand{\Wnkh}{\sW_\field(n)_h}

\newcommand\bsg{{\boldsymbol\sigma}}

\newcommand{\bse}{\boldsymbol{e}}

\newcommand{\bfj}{{\mathbf{j}}}

\newcommand{\bfl}{{\mathbf{0}}}

\newcommand{\bft}{{\mathbf{t}}}

\newcommand{\msK}{\mathscr K}
\newcommand{\msKnh}{\mathscr K(n)_h}
\newcommand{\msKnhq}{\mathscr K'(n)_h}
\newcommand{\msKnoq}{\mathscr K'(n)_1}
\newcommand{\Kn}{\mathcal K_\sZ(n)}
\newcommand{\hbfKn}{\h{{\mathcal K}}_{\sQ}(n)}
\newcommand{\hKn}{\h{{\mathcal K}}_\sZ(n)}
\newcommand{\hKnk}{\h{{\mathcal K}}_\field(n)}
\newcommand{\Knk}{\mathcal K_\field(n)}
\newcommand{\field}{\mathpzc k}

\def\sB{{\mathcal B}}

\def\sI{{\mathcal I}}

\def\sL{{\mathcal L}}

\def\sN{{\mathcal N}}

\def\sQ{{\mathcal Q}}

\def\sV{{\mathcal V}}
\def\sW{{\mathcal W}}

\def\sZ{{\mathcal Z}}

\newcommand{\mbzn}{\mathbb Z^{n}}
\newcommand{\mbnn}{\mathbb N^{n}}
\newcommand{\mbnnh}{\mathbb N_{lp^{h-1}}^{n}}
\newcommand{\mbn}{\mathbb N}
\newcommand{\mbq}{\mathbb Q}

\newcommand{\mbz}{\mathbb Z}

\newcommand{\ttu}{\mathtt{u}}
\newcommand{\tts}{\mathtt{s}}

\newcommand{\spann}{\operatorname{span}}
\newcommand{\diag}{\operatorname{diag}}

\def\ro{\text{\rm ro}}
\def\co{\text{\rm co}}

\newcommand{\la}{{\lambda}}
\newcommand{\La}{\Lambda}
\newcommand{\ga}{{\gamma}}

\newcommand{\Th}{\Theta}
\newcommand{\dt}{\delta}

\newcommand{\up}{v}
\newcommand{\vi}{\varphi}
\newcommand{\vep}{\varepsilon}

\newcommand{\al}{\alpha}
\newcommand{\bt}{\beta}
\newcommand{\sg}{\sigma}

\def\th{\theta} 

\newcommand{\ol}{\overline}

\newcommand{\lb}{\overline{\lambda}}
\newcommand{\mb}{\overline{\mu}}

\newcommand{\lan}{\langle}
\newcommand{\ran}{\rangle}

\newcommand{\leb}{\left[}
\newcommand{\bbl}{\big[}
\newcommand{\bbr}{\big]}

\newcommand{\rib}{\right]}

\def\ggp#1#2{\left[\kern-3.2pt\left[{#1\atop #2}\right]\kern-3.2pt\right]}

\newcommand{\p}{\prec}
\newcommand{\pr}{\preccurlyeq}
\def\leq{\leqslant}\def\geq{\geqslant}

\newcommand{\ot}{\otimes}

\newcommand{\bin}{\bigcup}

\newcommand{\han}{\subseteq}

\newcommand{\h}{\widehat}
\newcommand{\ti}{\widetilde}

\newcommand{\Lanr}{\Lambda(n,r)}
\newcommand{\tiThn}{\ti\Th(n)}
\newcommand{\Thn}{\Th(n)}

\newcommand{\Thnrh}{\Th(n,r)_h}
\newcommand{\tiThnh}{\ti\Th(n)_h}
\newcommand{\tiThnhq}{\ti\Th'(n)_h}
\newcommand{\Thnpm}{\Th^\pm(n)}
\newcommand{\Thnpmh}{\Th^\pm(n)_h}
\newcommand{\Thnp}{\Th^+(n)}
\newcommand{\Thnm}{\Th^-(n)}
\newcommand\Thnr{\Theta(n,r)}

\newcommand{\ra}{\rightarrow}
\newcommand{\map}{\mapsto}

\newcommand{\dzr}{\dot\zeta_r}

\newcommand{\hdzrk}{\h{\dot\zeta}_{r,\field}}

\newcommand{\zrk}{\zeta_{r,\field}}

\newcommand{\tong}{\stackrel{\thicksim}{\,\rightarrow}}

\newcommand{\snkh}{\tts_\field(n)_h}
\newcommand{\snkhr}{\tts_\field(n,r)_h}
\newcommand{\Unkhr}{\ti\ttu_\field(n,r)_h}

\newcommand{\Unkor}{\ti\ttu_\field(n,r)_1}
\newcommand{\bfUn}{{U}_{\sQ}(n)}

\newcommand{\Un}{U_\sZ(n)}
\newcommand{\Unk}{U_\field(n)}
\newcommand{\barUnkh}{{\ttu}_\field(n)_h}
\newcommand{\barUnko}{{\ttu}_\field(n)_1}
\newcommand{\Unkh}{\ti{\ttu}_\field(n)_h}
\newcommand{\Unko}{\ti{\ttu}_\field(n)_1}
\newcommand{\Unkt}{\ti{\ttu}_\field(n)_2}
\newcommand{\Unkhz}{\ti{\ttu}_\field^0(n)_h}

\newcommand{\Sr}{{\mathcal S}_\sZ(n,r)}
\newcommand{\Srk}{{\mathcal S}_\field(n,r)}

\newcommand{\bfSr}{{{\mathcal S}_\sQ}(n,r)}

\begin{document}
\title{BLM realization for Frobenius--Lusztig Kernels of type $A$}

\author{Qiang Fu}
\address{Department of Mathematics, Tongji University, Shanghai, 200092, China.}
\email{q.fu@tongji.edu.cn}


\thanks{Supported by the National Natural Science Foundation
of China, the Program NCET, Fok Ying Tung Education Foundation
 and the Fundamental Research Funds for the Central Universities}

\begin{abstract}
The infinitesimal quantum $\frak{gl}_n$ was realized in \cite[\S 6]{BLM}. We will realize
Frobenius--Lusztig Kernels of type $A$ in this paper.
\end{abstract}

 \sloppy \maketitle

\section{Introduction}
In 1990, Ringel discovered the Hall algebra realization \cite{R90}
of the positive part of the quantum enveloping algebras of finite type. Almost at the same time, the entire quantum $\frak{gl}_n$  was realized  by A. A. Beilinson, G. Lusztig and R. MacPherson in \cite{BLM}.
They first used $q$-Schur algebras to construct a $\mbq(v)$-algebra $\hbfKn$, and then proved that the quantum enveloping algebra of $\frak{gl}_n$ over $\mbq(v)$ can be realized as a subalgebra of $\hbfKn$.

Let  $\Unk$ be the the quantum enveloping algebra of $\frak{gl}_n$ over $\field$ with standard generators $E_i^{(m)}$, $F_i^{(m)}$, $K_i^{\pm 1}$ and
$\bbl {K_i;0 \atop t} \bbr$, where $\field$ is a commutative ring containing a primitive $l'$th root $\vep$ of $1$. Let $p=\text{char} \field$.  For $h\geq 1$, let $\Unkh$ be the $\field$-subalgebra of $\Unk$ generated by $E_{i}^{(m)}$, $F_i^{(m)}$, $K_j^{\pm 1}$, $\bbl{K_j;0\atop t}\bbr$ for $1\leq i\leq n-1$, $1\leq j\leq n$ and $0\leq m,t<lp^{h-1}$, where $l=l'$ if $l'$ is odd, and $l=l'/2$ otherwise. Then we have $\Unko\han\Unkt\han\cdots\han\Unk$.
In the case where $l'$ is an odd number, let
$\barUnkh=\Unkh/\lan K_1^l-1,\cdots,K_n^l-1\ran$. The algebra $\barUnko$ is called the infinitesimal quantum $\frak{gl}_n$ and the algebra $\barUnkh$ is called Frobenius--Lusztig Kernels of $\Unk$ (cf. \cite{Drupieski}). The algebra $\barUnko$ was realized in \cite[\S6]{BLM}. In this paper, we will realize the algebra $\barUnkh$ for all $h\geq 1$.
More precisely, we will first construct the $\field$-algebra $\msKnhq$ in \S4. Then we
will prove in \ref{realization} that $\barUnkh \cong\msKnhq$ in the case where $l'$ is odd, and that $\Unkh\cong \msKnhq$ in the case where $l'$ is even and $\field$ is a field.

Let $\Srk$ be the $q$-Schur algebra over $\field$. Certain subalgebra, denoted by $\Unkhr$, of $\Srk$ was constructed in \cite[\S4]{DFW12}.
It is proved in \cite{Fu05} that $\Unkor$ is isomorphic to the little $q$-Schur algebra introduced in \cite{DFW05,Fu07}. We will prove
in \ref{little q-Schur algebras} that the algebra $\Unkhr$ is a homomorphic image of $\Unkh$.

Infinitesimal $q$-Schur algebras are certain important subalgebras of $q$-Schur algebras (cf. \cite{DNP96,Cox,Cox00}). For $h\geq 1$ let $\snkh$ be the $\field$-subalgebra of $\Unk$ generated by the algebra $\Unkh$ and $\bbl{K_j;0\atop t}\bbr$ ($1\leq j\leq n$, $t\in\mbn$). We will prove in \ref{infinitesimal q-Schur algebras} that the infinitesimal $q$-Schur algebra $\snkhr$ is a homomorphic image of $\snkh$.


Throughout this paper, let $\sZ=\mathbb Z[\up,\up^{-1}]$ where $\up$ is an indeterminate and let $\sQ=\mbq(v)$ be the fraction field of $\sZ$.
For $i\in\mbz$ let $[i]=\frac{v^i-v^{-i}}{v-v^{-1}}$.
For integers $N,t$ with $t\geq 0$, let
\begin{equation*}
\leb{N\atop t}\rib=\frac{[N][N-1]\cdots[N-t+1]}{[t]^!}\in\sZ
\end{equation*}
where $[t]^{!}=[1][2]\cdots[t]$.
For $\mu\in\mbzn$ and $\la\in\mbnn$ let $\bbl{\mu\atop\la}\bbr=\bbl{\mu_1\atop\la_1}\bbr\cdots\bbl{\mu_n\atop\la_n}\bbr.$

Let $\field$ be a commutative ring containing a primitive $l'$th root $\vep$ of $1$ with $l'\geq 1$.
Let $l\geq 1$ be defined by
$$l=
\begin{cases}
l'&\text{if $l'$ is
odd},\\
l'/2&\text{if $l'$ is even}.
\end{cases}$$
Let $p$ be the characteristic of $\field$.
We will regard $\field$ as a  $\sZ$-module by specializing
$\up$ to $\vep$. When $v$ is specialized to $\vep$,
$\big[{c\atop t}\big]$
specialize to the element
$\big[{c\atop t}\big]_\vep$ in $\field$.

\section{The BLM construction of quantum $\frak{gl}_n$}

Following \cite{Ji} we define
the quantum enveloping algebra $\bfUn$ of ${\frak {gl}}_n$ to be the $\mbq(v)$-algebra
with generators
$$E_i,\ F_i\quad(1\leq i\leq n-1),\  K_j,\  K_j^{-1}\quad(1\leq j\leq n)$$
and relations

$(a)\ K_{i}K_{j}=K_{j}K_{i},\ K_{i}K_{i}^{-1}=1;$

$(b)\ K_{i}E_j=v^{\dt_{i,j}-\dt_{i,j+1}} E_jK_{i};$

$(c)\ K_{i}F_j=v^{\dt_{i,j+1}-\dt_{i,j}} F_jK_i;$

$(d)\ E_iE_j=E_jE_i,\ F_iF_j=F_jF_i\ when\ |i-j|>1;$

$(e)\ E_iF_j-F_jE_i=\delta_{i,j}\frac {\widetilde
K_{i}-\widetilde K_{i}^{-1}}{v-v^{-1}},\ where \
\widetilde K_i =K_{i}K_{i+1}^{-1};$

$(f)\
E_i^2E_j-(v+v^{-1})E_iE_jE_i+E_jE_i^2=0\
 when\ |i-j|=1;$

$(g)\
F_i^2F_j-(v+v^{-1})F_iF_jF_i+F_jF_i^2=0\
 when\ |i-j|=1.$

Following \cite{Lu90}, let $\Un$ be the $\sZ$-subalgebra of $\bfUn $
generated by all $E_i^{(m)}$, $F_i^{(m)}$, $K_i^{\pm 1}$ and
$\bbl {K_i;0 \atop t} \bbr$, where for $m,t\in\mathbb N$,
$$E_i^{(m)}=\frac{E_i^m}{[m]^!},\,\,F_i^{(m)}=\frac{F_i^m}{[m]^!},\text{ and }
\bigg[ {K_i;0 \atop t} \bigg] =
\prod_{s=1}^t \frac
{K_iv^{-s+1}-K_i^{-1}v^{s-1}}{v^s-v^{-s}}.$$

Let $\Thn$ be the set of all $n\times n$ matrices over
$\mathbb N$. Let $\Thnpm$ be the set of all $A\in\Thn$ whose diagonal entries are zero. Let $\Thnp$ (resp.
$\Thnm$) be the subset of $\Thn$ consisting of those
matrices $(a_{i,j})$ with $a_{i,j}=0$ for all $i\geq j$ (resp.
$i\leq j$). For $A\in\Thnpm$, write $A=A^++A^-$ with
$A^+\in\Thnp$ and $A^-\in\Thnm$.  For
$A\in\Thnpm$ let
\begin{equation*}
E^{(A^+)}=\prod_{i\leq s<j\atop 1\leq i,j\leq n}
E_{s}^{(a_{ij})},\quad F^{(A^-)}=\prod_{j\leq s<i\atop
1\leq i,j\leq n}F_{s}^{(a_{i,j})}
\end{equation*}
where the ordering of the products is the same as in
\cite[3.9]{BLM}. According to \cite[4.5]{Lu90} and \cite[7.8]{Lu901} we have the following result.
\begin{Prop}\label{basis for Un}
The set $$\{E^{(A^+)}\prod_{1\leq i\leq n}K_i^{\dt_i}\leb{K_i;0\atop\la_i}\rib F^{(A^-)}\mid A\in\Thnpm,\,\dt,\la\in\mbnn,\,\dt_i\in\{0,1\},\,\forall i\}$$ forms a $\sZ$-basis of $\Un$.
\end{Prop}

Using the stabilization property of the multiplication of $q$-Schur
algebras, an important algebra $\Kn$ over $\sZ$ (without
1), with basis $\{[A]\mid A\in\tiThn\}$ was constructed in
\cite[4.5]{BLM}, where
$\tiThn=\{(a_{ij})\in M_n(\mbz)\mid a_{ij}\geq 0\,\,\forall
1\leq i\neq j\leq n\}$.

Following \cite[5.1]{BLM}, let $\hbfKn$ be
the vector space of all formal $\mathbb
Q(\up)$-linear combinations
$\sum_{A\in\tiThn}\beta_A[A]$ satisfying
the following property:
for any ${\bf x}\in\mathbb Z^n$,

\begin{equation}\label{(F)}
\text{the sets}
\aligned
&\{A\in\tiThn\ |\ \beta_A\neq0,\ \ro(A)={\bf
x}\}\\
&\{A\in\tiThn\ |\ \beta_A\neq0,\ \co(A)={\bf
x}\}\endaligned
\text{ are finite,}
\end{equation}
where $\ro(A)=(\sum_ja_{1,j},\cdots,\sum_ja_{n,j})$ and
$\co(A)=(\sum_ia_{i,1},\cdots,\sum_ia_{i,n})$ are the sequences of
row and column sums of $A$. The product of two elements
$\sum_{A\in\tiThn}\beta_A[A]$,
$\sum_{B\in\tiThn}\gamma_B[B]$ in $\hbfKn$ is defined to be
$\sum_{A,B}\beta_A\gamma_B[A]\cdot[B]$ where $[A]\cdot[B]$ is the
product in $\Kn$. Then  $\hbfKn$ becomes an associative algebra over $\mbq(\up)$.

For $A\in\Thnpm$, $\dt\in\mbzn$ and $\la\in\mbnn$
let
\begin{equation*}
\begin{split}
A(\dt,\la)&=\sum_{\mu\in\mbzn}v^{\mu\centerdot\dt}
\leb{\mu\atop\la}\rib[A+\diag(\mu)]\in\hbfKn;\\
A(\dt)&=\sum_{\mu\in\mbzn}v^{\mu\centerdot\dt}
[A+\diag(\mu)]\in\hbfKn,
\end{split}
\end{equation*}
where $\mu\centerdot\dt=\sum_{1\leq i\leq n}\mu_i\dt_i$.

The next result is proved in \cite[5.5,5.7]{BLM}.

\begin{Thm}\label{BLM realization of bfUn}
There is an injective algebra homomorphism $\vi:\bfUn \ra \hbfKn$
satisfying
$$E_i\map E_{i,i+1}(\mathbf 0),\ K_1^{j_1}K_2^{j_2}\cdots
K_n^{j_n}\mapsto 0(\mathbf j),\ F_i\map E_{i+1,i}(\mathbf 0).$$
Furthermore the set $\{A({\bf j})\ |\ A\in\Thnpm,\ {\bf j}\in\mathbb
Z^n\}$ forms a $\mbq(v)$-basis for $\vi(\bfUn)$.
\end{Thm}

We shall identify $\bfUn$ with $\vi(\bfUn)$.
According to \cite[4.2,4.3,4.4]{Fu12}, we have the following result.

\begin{Prop}\label{basis2 for Un}
The algebra $\Un$ is generated as a $\sZ$-module by the elements
$A(\dt,\la)$ for $A\in\Thnpm$, $\dt\in\mbzn$ and $\la\in\mbnn$. Furthermore, each of the following set forms a $\sZ$-basis for $\Un:$

$(1)$ $ \{A(\bfl)0(\dt,\la)\mid A\in\Thnpm,\,
\dt,\la\in\mbnn,\,\dt_i\in\{0,1\},\forall i\};$

$(2)$ $ \{A(\dt,\la)\mid A\in\Thnpm,\,
\dt,\la\in\mbnn,\,\dt_i\in\{0,1\},\forall i\}.$
\end{Prop}

We end this section by recalling an important triangular relation in $\Kn$.
For
$A=(a_{s,t})\in\tiThn$  let
$$\sg_{i,j}(A)
=\begin{cases}
\sum_{s\leq i;t\geq j}a_{s,t}&\text{if $i<j$}\\
\sum_{s\geq i;t\leq j}a_{s,t}&\text{if $i>j$}.
\end{cases}$$
Following \cite{BLM}, for $A,B\in\tiThn$, define
$B \pr A$ if and only if $\sg_{i,j}( B )\leq\sg_{i,j}(A)$ for all
$i\not=j$. Put $ B \p A$ if $ B \pr A$ and $\sg_{i,j}(
B )<\sg_{i,j}(A)$ for some $i\not=j$.

According to \cite[5.5(c)]{BLM}, for $A\in\Thnpm$
and $\la\in\mbzn$ the following triangular relation holds in
$\Kn$:
\begin{equation}\label{Tri}
E^{(A^+)}[\diag(\la)]F^{(A^-)}=[A+\diag(\la-\bsg(A))]+f
\end{equation}
where $\bsg(A)=(\sg_1(A),\cdots,\sg_n(A))$ with $\sg_i(A)=\sum_{j<
i}(a_{i,j}+a_{j,i})$ and $f$ is a finite $\sZ$-linear combination
of $[B]$ with $B\in\tiThn$ such that $B\p A$.

\section{The algebra $\Unkh$}

Let $\Unk=\Un\ot_\sZ\field$. We shall denote the images of $E_i^{(m)}$, $F_i^{(m)}$, $A(\dt,\la)$, etc. in $\Unk$ by the same letters. For $h\geq 1$ let $\Unkh$ be the $\field$-subalgebra of $\Unk$ generated by the elements $E_{i}^{(m)}$, $F_i^{(m)}$, $K_j^{\pm 1}$, $\bbl{K_j;0\atop t}\bbr$ for $1\leq i\leq n-1$, $1\leq j\leq n$ and $0\leq m,t<lp^{h-1}$. If $l'$ is an odd number, we let
\begin{equation}\label{Unkh}
\barUnkh=\Unkh/\lan K_1^l-1,\cdots,K_n^l-1\ran.
\end{equation}
The algebra $\barUnkh$ is called Frobenius--Lusztig Kernels of $\Unk$.  We will construct several $\field$-bases for $\Unkh$ in \ref{basis of Unkh}.

We need some preparation before proving \ref{basis of Unkh}.

\begin{Lem}\label{m t}
Let $m=m_0+lm_1$, $0\leq m_0 \leq l-1$, $m_1 \in\mbn$. Then $$\bigg[{m\atop t}\bigg]_\vep=\vep^{l(t_1l-t_1m_0-tm_1)}\bigg[{m_0\atop t_0}\bigg]_\vep\bigg({m_1\atop t_1}\bigg)$$ for $0\leq t\leq m$, where
 $t=t_0+lt_1$ with $0\leq t_0\leq l-1$ and $t_1\in\mbn$.
\end{Lem}

\begin{Lem}\label{identity}
The following identity hold in the field $\field:$
$\big({m+p^{h-1}\atop s}\big)=\big({m\atop s}\big)$
for $m\in\mbz$ and $0\leq s<p^{h-1}$.
\end{Lem}
\begin{proof}
We consider the polynomial ring $\field[x,y]$.
Since the characteristic of $\field$ is $p$ we see that
$$\sum_{0\leq j\leq p^{h-1}}\bigg({p^{h-1}\atop j}\bigg)x^jy^{p^{h-1}-j}=(x+y)^{p^{h-1}}=x^{p^{h-1}}+y^{p^{h-1}}.$$
It follows that $\big({p^{h-1}\atop j}\big)=0$ for $0<j<p^{h-1}$.
This implies that
$$\bigg({m+p^{h-1}\atop s}\bigg)=
\sum_{0\leq j\leq s}\bigg({p^{h-1}\atop j}\bigg)\bigg({m\atop s-j}\bigg)=
\bigg({m\atop s}\bigg)$$
for $m\in\mbz$ and $0\leq s<p^{h-1}$.
\end{proof}

We now generalize \ref{identity} to the quantum case.
\begin{Lem}\label{Gauss identity1}
Assume $0\leq a<lp^{h-1}$ and $b\in\mbz$. Then we have
$\bbl{b+lp^{h-1}\atop a}\bbr_\vep=\vep^{-alp^{h-1}}\bbl{b\atop a}\bbr_\vep$. In particular, we have $\bbl{b+l'p^{h-1}\atop a}\bbr_\vep=\bbl{b\atop a}\bbr_\vep$
\end{Lem}
\begin{proof}
We write $a=a_0+a_1l$ and $b=b_0+b_1l$ with $0\leq a_0,b_0< l$, $a_1\in\mbn$ and $b_1\in\mbz$.
If $b\in\mbn$, then by \ref{m t} and \ref{identity} we conclude that
\begin{equation*}
\begin{split}
\leb{b+lp^{h-1}\atop a}\rib_\vep&=\vep^{-alp^{h-1}}\vep^{l(a_1l-a_1b_0-a_1b_1l-a_0b_1)}
\leb{b_0\atop a_0}\rib_\vep\bigg({b_1+p^{h-1}\atop a_1}\bigg)\\
&=\vep^{-alp^{h-1}}\vep^{l(a_1l-a_1b_0-a_1b_1l-a_0b_1)}
\leb{b_0\atop a_0}\rib_\vep\bigg({b_1\atop a_1}\bigg)\\
&=\vep^{-alp^{h-1}}\leb{b\atop a}\rib_\vep.
\end{split}
\end{equation*}
Furthermore if $b+lp^{h-1}<0$, then $-b+a-1-lp^{h-1}\geq 0$ and hence
\begin{equation*}
\begin{split}
\leb{b+lp^{h-1}\atop a}\rib_\vep&=(-1)^a\leb{-b+a-1-lp^{h-1}\atop a}\rib_\vep
=(-1)^a\vep^{alp^{h-1}}\leb{-b+a-1\atop a}\rib_\vep=\vep^{-alp^{h-1}}\leb{b\atop a}\rib_\vep.
\end{split}
\end{equation*}
Now we assume $-lp^{h-1}\leq b<0$. According to \ref{m t} we have
\begin{equation}\label{eq1}
\leb{b+lp^{h-1}\atop a}\rib_\vep=\vep^{-alp^{h-1}}\vep^{l(a_1l-a_1b_0-ab_1)}\leb{b_0\atop a_0}\rib_\vep\bigg({b_1\atop a_1}\bigg).
\end{equation}
If $ a_0-b_0-1\geq 0$ then
$\bbl{b_0\atop a_0}\bbr_\vep=(-1)^{a_0}\bbl{a_0-b_0-1\atop a_0}\bbr_\vep=0$ and hence, by \ref{m t} and \eqref{eq1}, we have
\begin{equation*}
\begin{split}
\leb{b\atop a}\rib_\vep
&=
(-1)^a\leb{l(a_1-b_1)+(a_0-b_0-1)\atop a}\rib_\vep\\
&=(-1)^a\vep^{l(a_1l-a_1(a_0-b_0-1)-a(a_1-b_1))}\leb{a_0-b_0-1\atop a_0}\rib_\vep
\bigg({a_1-b_1\atop a_1}\bigg)\\
&=0\\
&=\vep^{alp^{h-1}}\leb{b+lp^{h-1}\atop a}\rib_\vep.
\end{split}
\end{equation*}
Now we assume $-lp^{h-1}\leq b<0$ and $ a_0-b_0-1< 0$. Then $a_1-b_1-1\geq 0$ and $0\leq l+a_0-b_0-1<l$. According to \ref{m t} we have
\begin{equation*}
\begin{split}
\leb{b\atop a}\rib_\vep&=(-1)^a\leb{-b+a-1\atop a}\rib_\vep\\
&=(-1)^a\leb{l(a_1-b_1-1)+(l+a_0-b_0-1)\atop a}\rib_\vep\\
&=(-1)^a\vep^{l(-a_1(a_0-b_0-1)-a(a_1-b_1-1))}\leb{l+a_0-b_0-1\atop a_0}\rib_\vep
\bigg({a_1-b_1-1\atop a_1}\bigg)\\
&=(-1)^{a_1l+a_1}\vep^{l(-a_1(a_0-b_0-1)-a(a_1-b_1-1))}\leb{b_0-l\atop a_0}\rib_\vep
\bigg({b_1\atop a_1}\bigg).
\end{split}
\end{equation*}
Since $0\leq a_0<l$ and $[m+l]_\vep=\vep^{-l}[m]_\vep$ we see that
$\bbl{b_0-l\atop a_0}\bbr=\vep^{a_0l}\bbl{b_0\atop a_0}\bbr_\vep$. This implies that
\begin{equation}\label{eq2}
\leb{b\atop a}\rib_\vep=(-1)^{a_1l+a_1}\vep^{l(a_0-a_1(a_0-b_0-1)-a(a_1-b_1-1))}\leb{b_0\atop a_0}\rib_\vep
\bigg({b_1\atop a_1}\bigg).
\end{equation}
Furthermore since $\vep^{2l}=1$ and $(a_1^2l-a_1)-(a_1l+a_1)=-2a_1+la_1(a_1-1)$ is even, we see that
\begin{equation*}
\begin{split}
\frac{\vep^{l(a_1l-a_1b_0-ab_1)}}{\vep^{l(a_0-a_1(a_0-b_0-1)-a(a_1-b_1-1))}}
&=\vep^{l(-2ab_1-2a_1b_0-2a_0+2a_0a_1)}\vep^{l(a_1^2l-a_1)}\\
&=\vep^{l(a_1^2l-a_1)}=\vep^{l(a_1l+a_1)}=(-1)^{a_1(l+1)}.
\end{split}
\end{equation*}
Thus by \eqref{eq1} and \eqref{eq2} we conclude that
$\bbl{b+lp^{h-1}\atop a}\bbr_\vep=\vep^{-alp^{h-1}}\bbl{b\atop a}\bbr_\vep$. The proof is completed.
\end{proof}
\begin{Coro}\label{Gauss identity2}
Assume $0\leq a,b<lp^{h-1}$ and $a+b\geq lp^{h-1}$. Then $\bbl{a+b\atop a}\bbr_\vep=0$.
\end{Coro}
\begin{proof}
According to \ref{Gauss identity1} we have
$\bbl{a+b\atop a}\bbr_\vep=\vep^{-alp^{h-1}}\bbl{a+b-lp^{h-1}\atop a}\bbr_\vep$.
Since $0\leq a+b-lp^{h-1}<a$, we see that $\bbl{a+b-lp^{h-1}\atop a}\bbr_\vep=0$. The assertion follows.
\end{proof}

Let $\Unkhz$ be the $\field$-subalgebra of $\Unkh$ generated by $K_j^{\pm 1}$, $\bbl{K_j;0\atop t}\bbr$ for $1\leq j\leq n$ and $0\leq t<lp^{h-1}$.
For $h\geq 1$ let
$$\mbnnh=\{\la\in\mbnn\mid 0\leq\la_i<lp^{h-1},\,\forall i\}.$$

\begin{Lem}\label{basis B0}
The set $\frak M^0=\{\prod_{1\leq i\leq n}K_i^{\dt_i}\bbl{K_i;0\atop \la_i}\bbr\mid \dt\in\mbnn,\,\dt_i\in\{0,1\},\,\la\in\mbnnh\}$ forms
a $\field$-basis for $\Unkhz$.
\end{Lem}
\begin{proof}
Let $V_1=\spann_\field\frak M^0$. From \ref{basis for Un}, we see that the set $\frak M^0$ is linearly independent. Thus it is enough to prove that $\Unkhz=V_1$. Let $V_2$
be the $\field$-submodule of $\Unkhz$ spanned by the elements $\prod_{1\leq i\leq n}K_i^{\dt_i}\bbl{K_i;0\atop \la_i}\bbr$ ($\dt\in\mbzn$, $\la\in\mbnn$, $0\leq\la_i<lp^{h-1}$, for all $i$).
According to \cite[2.3(g8)]{Lu90},
for $0\leq t,t'<lp^{h-1}$ we have
$$\vep^{t't}\leb{K_i;0\atop t'}\rib
\leb{K_i;0\atop t}\rib=\leb{t+t'\atop t}\rib_\vep\leb{K_i;0\atop t+t'}\rib-\sum_{0<j\leq t'}(-1)^j\vep^{t(t'-j)}\leb{t+j-1\atop j}\rib_\vep K_i^j
\leb{K_i;0\atop t'-j}\rib
\leb{K_i;0\atop t}\rib.
$$
Note that by \ref{Gauss identity2} we have
$\bbl{t+t'\atop t}\bbr_\vep\bbl{K_i;0\atop t+t'}\bbr=0$
for $0\leq t,t'<lp^{h-1}$ with $t+t'\geq lp^{h-1}$. Thus, by induction on $t'$ we see that
$\bbl{K_i;0\atop t'}\bbr
\bbl{K_i;0\atop t}\bbr\in V_2$ for $0\leq t,t'<lp^{h-1}$. It follows that
$\Unkhz=V_2$. Furthermore, by the  proof of \cite[2.14]{Lu90}, for $m\geq 0$ and $0\leq t<lp^{h-1}$ we have
\begin{equation*}
\begin{split}
K_i^{m+2}\leb{K_i;0\atop t}\rib&=
\vep^t(\vep^{t+1}-\vep^{-t-1})K_i^{m+1}\leb{K_i;0\atop t+1}\rib
+\vep^{2t}K_i^{m}\leb{K_i;0\atop t}\rib,\\
K_i^{-m-1}\leb{K_i;0\atop t}\rib&=
-\vep^{-t}(\vep^{t+1}-\vep^{-t-1})K_i^{-m}\leb{K_i;0\atop t+1}\rib
+\vep^{-2t}K_i^{-m+1}\leb{K_i;0\atop t}\rib.
\end{split}
\end{equation*}
If $t+1=lp^{h-1}$, then $\vep^t(\vep^{t+1}-\vep^{-t-1})K_i^{m+1}\bbl{K_i;0\atop t+1}\bbr
=-\vep^{-t}(\vep^{t+1}-\vep^{-t-1})K_i^{-m}\bbl{K_i;0\atop t+1}\bbr=0$. Thus by induction on $m\geq 0$ we see that $K_i^{\pm m}\bbl{K_i;0\atop t}\bbr\in V_1$ for $0\leq t<lp^{h-1}$. This implies that $V_1=V_2$. The assertion follows.
\end{proof}

We are now ready to prove \ref{basis of Unkh}.
Let $\Thnpmh=\{A\in\Thnpm\mid 0\leq a_{s,t}<lp^{h-1},\,\forall s\not=t\}$.
\begin{Lem}\label{span set of tiUnkh}
The algebra $\Unkh$ is generated as a $\field$-module by the elements
$A(\dt,\la)$ for $A\in\Thnpmh$, $\dt\in\mbzn$ and $\la\in\mbnnh$.
\end{Lem}
\begin{proof}
Let $V_h$ be the $\field$-submodule of $\Unk$ spanned by
$A(\dt,\la)$ for $A\in\Thnpmh$, $\dt\in\mbzn$ and $\la\in\mbnnh$.
According to \cite[3.5(1)]{Fu12} for $A\in\Thnpmh$, $0\leq m<lp^{h-1}$, $1\leq i\leq n-1$, $\dt\in\mbzn$ and $\la\in\mbnnh$,  we have
\begin{equation*}
\begin{split}
&\qquad(mE_{i,i+1})(\bfl)A(\dt,\la)\\
&=\sum_{\tiny{\begin{array}{c}\bft\in\La(n,m),\,0\leq j\leq\la_i\\
t_u\leq
a_{i+1,u},\,\forall u\not=i+1\\
0\leq k\leq\la_{i+1},\,
0\leq c\leq\min
\{t_i,j\}\end{array}}}
f^\bft_{j,c,k}
\bigg(A+\sum_{u\not=i}t_uE_{i,u}-\sum_{u\not=i+1}t_uE_{i+1,u}\bigg)
(\dt+\al^\bft_{j,c,k},\la+\bt^\bft_{j,c,k}).
\end{split}
\end{equation*}
where
$\al^\bft_{j,c,k}=
\big(\sum_{i>u}t_u+\la_i-j-c\big)\bse_i+\big(\la_{i+1}-k-\sum_{i+1>u}t_u
\big)
\bse_{i+1},$ $
\bt^\bft_{j,c,k}=(t_i+j-c-\la_i)\bse_i+(k-\la_{i+1})\bse_{i+1}$ with
$\bse_i=(0,\cdots,0,\underset i1,0\cdots,0)\in\mbn^n$,
and $$f^\bft_{j,c,k}=\vep^{g^\bft_{j,k}}\prod_{u\not=i}{\leb{a_{i,u}+t_u\atop t_u}\rib}_\vep
\leb{-t_i\atop\la_i-j}\rib_\vep\leb{t_i+j-c\atop t_i}\rib_\vep\leb{t_i\atop c}\rib_\vep
\leb{t_{i+1}\atop\la_{i+1}-k}\rib_\vep$$
with $g^\bft_{j,k}=\sum_{j> u,\,j\not=i}a_{i,j}t_u-\sum_{j>u,\,
j\not=i+1}a_{i+1,j}t_u+\sum_{u'\not=i,i+1,\,u<u'}t_ut_{u'}-t_i\dt_i
+t_{i+1}\dt_{i+1}+2jt_i-kt_{i+1}$.
If $A+\sum_{u\not=i}t_uE_{i,u}-\sum_{u\not=i+1}t_uE_{i+1,u}\not\in\Thnpmh$ then $a_{i,u}+t_u\geq lp^{h-1}$ for some $u\not=i$. From \ref{Gauss identity2} we see that
${\bbl{a_{i,u}+t_u\atop t_u}\bbr}_\vep=0$ and hence $f^\bft_{j,c,k}=0$. Furthermore, if
$\la+\bt^\bft_{j,c,k}\not\in\mbnnh$ then $(\la+\bt_{j,c,k}^t)_i=t_i+j-c\geq lp^{h-1}$. From \ref{Gauss identity2} we see that $\bbl{t_i+j-c\atop t_i}\bbr_\vep=0$ and hence $f^\bft_{j,c,k}=0$.
Thus we conclude that
\begin{equation}\label{han1}
(mE_{i,i+1})(\bfl)V_h\han V_h,
\end{equation}
for $0\leq m<lp^{h-1}$ and $1\leq i\leq n-1$. Similarly, using \cite[3.4,3.5(2)]{Fu12} we see that
\begin{equation}\label{han2}
(mE_{i+1,i})(\bfl)V_h\han V_h \text{ and }0(\ga,\mu)V_h\han V_h
\end{equation}
for $0\leq m<lp^{h-1}$, $1\leq i\leq n-1$, $\ga\in\mbzn$ and $\mu\in\mbnnh$.
Combining \eqref{han1} with \eqref{han2} implies that
\begin{equation}\label{han}
 \Unkh\han\Unkh V_h\han V_h.
\end{equation}

On the other hand, from \cite[3.4]{Fu12}  we see that for $A\in\Thnpmh$, $\dt\in\mbzn$ and $\la\in\mbnnh$,
\begin{equation}
A(\bfl)0(\dt,\la)=\vep^{\co(A)\centerdot(\dt+\la)}A(\dt,\la)+
\sum_{\bfj\in\mbnn,\,\bfl<\bfj\leq\la}
\vep^{\co(A)\centerdot(\dt+\la-\bfj)}\leb{\co(A)\atop\bfj}\rib
A(\dt-\bfj,\la-\bfj).
\end{equation}
This implies that
\begin{equation}\label{span}
V_h=\spann_\field\{A(\bfl)0(\dt,\la)\mid A\in\Thnpmh,\,\dt\in\mbzn,\,\la\in\mbnnh\}.
\end{equation}
Furthermore, combining \eqref{Tri} with \ref{basis2 for Un} shows that for $A\in\Thnpmh$, $\dt\in\mbzn$ and $\la\in\mbnnh$,
$$E^{(A^+)}F^{(A^-)}\prod_{1\leq i\leq n}K_i^{\dt_i}\leb{K_i;0\atop\la_i}\rib=E^{(A^+)}F^{(A^-)}0(\dt,\la)=A(\bfl)0(\dt,\la)+f$$
where $f$ is a $\field$-linear combination of $B(\bfl)0(\ga,\mu)$ with $B\in\Thnpm$, $B\p A$, $\ga\in\mbzn$ and $\mu\in\mbnn$. From \eqref{han} and \eqref{span} we see that
$f$ must be a $\field$-linear combination of $B(\bfl)0(\ga,\mu)$ with $B\in\Thnpmh$, $B\p A$, $\ga\in\mbzn$ and $\mu\in\mbnnh$. Thus we conclude that
\begin{equation}\label{span1}
V_h=\spann_\field\big\{E^{(A^+)}F^{(A^-)}\prod_{1\leq i\leq n}K_i^{\dt_i}\leb{K_i;0\atop\la_i}\rib\mid A\in\Thnpmh,\,\dt\in\mbzn,\,\la\in\mbnnh\big\}\han\Unkh.
\end{equation}
The assertion follows.
\end{proof}
\begin{Prop}\label{basis of Unkh}
Each of the following set forms a $\field$-basis for $\Unkh:$
\begin{itemize}
\item[(1)]
$\frak M:=\{E^{(A^+)}\prod_{1\leq i\leq n}K_i^{\dt_i}\bbl{K_i;0\atop\la_i}\bbr F^{(A^-)}\mid A\in\Thnpmh,\,\dt\in\mbnn,\,\dt_i\in\{0,1\},\,\forall i,\,\la\in\mbnnh\};$
\item[(2)]
$\frak B:=\{A(\dt,\la)\mid A\in\Thnpmh,\,\dt\in\mbnn,\,\dt_i\in\{0,1\},\,\forall i,\,\la\in\mbnnh\};$
\item[(3)]
$\frak B':=\{A(\bfl)0(\dt,\la)\mid A\in\Thnpmh,\,\dt\in\mbnn,\,\dt_i\in\{0,1\},\,\forall i,\,\la\in\mbnnh\}.$
\end{itemize}
\end{Prop}
\begin{proof}
According to \ref{basis for Un} and \ref{basis2 for Un}, it is enough to prove that
$\Unkh=\spann_\field\frak M
=\spann_\field\frak B=\spann_\field\frak B'$.
From \ref{basis B0}, \ref{span set of tiUnkh}, \eqref{span} and \eqref{span1} we see that
$\Unkh=\spann_\field\frak M=\spann_\field\frak B'$.
For $A\in\Thnpmh$, $\dt\in\mbzn$ and $\la\in\mbnnh$ we have
\begin{equation*}
\begin{split}
A(\dt,\la)&=\vep^{\la_i}(\vep^{\la_i+1}-v^{-\la_i-1})A(\dt-\bse_i,\la+\bse_i)
+\vep^{2\la_i}A(\dt-2\bse_i,\la)\\
&=-\vep^{-\la_i}(\vep^{\la_i+1}-\vep^{-\la_i-1})A(\dt+\bse_i,\la+\bse_i)
+\vep^{-2\la_i}A(\dt+2\bse_i,\la)
\end{split}
\end{equation*}
Note that if $\la_i+1=lp^{h-1}$ then $\vep^{\la_i}(\vep^{\la_i+1}-v^{-\la_i-1})A(\dt-\bse_i,\la+\bse_i)=
-\vep^{-\la_i}(\vep^{\la_i+1}-\vep^{-\la_i-1})A(\dt+\bse_i,\la+\bse_i)=0$.
This together with \ref{span set of tiUnkh} shows that $\Unkh=\spann_\field\frak B$.
\end{proof}

\section{The algebra $\msKnhq$}
We will construct the algebra $\msKnhq$ in this section. We will prove in \ref{realization}  the algebra
$\msKnhq$ is the realization of $\Unkh$.

Let $\Knk=\Kn\ot_\sZ\field$, where $\field$ is regarded as a $\sZ$-module by specializing
$v$ to $\vep$. For $A\in\tiThn$ let
$$[A]_\vep=[A]\ot 1\in\Knk.$$
Let $\tiThnh$ be the set of all $A=(a_{i,j})\in\tiThn$ such that $a_{i,j}<lp^{h-1}$ for all $i\not=j$. We will denote by $\msKnh$ the $\field$-submodule of $\Knk$ spanned by the elements $[A]_\vep$ with $A\in\tiThnh$.

To construct the algebra $\msKnhq$ we need the following lemma (cf. \cite[6.2]{BLM} and \cite[5.1]{Fu07}).

\begin{Lem}\label{the property of K}
$(1)$ $\msKnh$ is a subalgebra of $\Knk$. It is generated by $[mE_{h,h+1}+\diag(\la)]_\vep$ and $[mE_{h+1,h}+\diag(\la)]_\vep$ for $0\leq m<lp^{h-1}$, $1\leq h\leq n-1$ and $\la\in\mbzn$.

$(2)$ Let $D$ be any diagonal matrix in $\tiThn$. The map
$\tau_{D}:\msKnh\ra\msKnh$ given by $[A]_\vep\ra[A+l'p^{h-1}D]_\vep$ is an
algebra homomorphism.
\end{Lem}
\begin{proof}
Let $A=(a_{s,t})\in\tiThnh$ and $0\leq m<lp^{h-1}$. Assume
that $B=(b_{s,t})\in\tiThnh$ is such that $B-mE_{i,i+1}$ is a diagonal matrix
such that $\co(B)=\ro(A)$.
According to \cite[4.6(a)]{BLM} we have
$$[B]_\vep\cdot[A]_\vep=\sum_{\bft\in\La(n,m)\atop \forall
u\not=i+1,t_u\leq
a_{i+1,u}}\vep^{\bt(\bft,A)}\prod_{1\leq u\leq n}{\leb{a_{i,u}+t_u\atop
t_u}\rib}_\vep \biggl[
A+\sum_{1\leq u\leq n}t_u( E_{i,u}- E_{i+1,u})\biggr]_\vep$$
where $\bt(\bft,A)=\sum_{j>
u}a_{i,j}t_u-\sum_{j>u}a_{i+1,j}t_u+\sum_{u<u'}t_ut_{u'}$.
Assume that
$A+\sum_ut_u(E_{i,u}-E_{i+1,u})\not\in\tiThnh$ for some $\bft$; then $a_{i,u}+t_u\geq
lp^{h-1}$ for some $u\neq i$. Since $0\leq a_{i,u},t_u<lp^{h-1}$, by \ref{Gauss identity2}, we conclude that
${\bbl{ a_{i,u}+t_u\atop t_u}\bbr}_\vep=0$ and
hence $[B]_\vep\cdot [A]_\vep\in\msKnh$. Similarly, we have
$[C]_\vep\cdot[A]_\vep\in\msKnh$, where $C$ is such that $C-mE_{i+1,i}$ is a
diagonal matrix such that $\co(C)=\ro(A)$.
Now using
\cite[4.6(c)]{BLM}, (1) can be proved in a  way similar to the proof of \cite[6.2]{BLM}.

According to \cite[4.6(a),(b)]{BLM} and \ref{Gauss identity1} we see that $\tau_{D}([A']_\vep[A]_\vep)=\tau_{D}([A']_\vep)
\tau_{D}([A]_\vep)$ for any $A'$ of the form $B,C$ as above. Since
$\msKnh$ is generated by elements like $[B]_\vep$, $[C]_\vep$ above, we conclude that
$\tau_D$ is an algebra homomorphism.
\end{proof}

Let $\tiThnhq$ be the set of all $n\times n$ matrices
$A=(a_{i,j})$ with $a_{i,j}\in\mathbb N,\ a_{i,j}<lp^{h-1}$ for all
$i\not= j$ and $a_{i,i}\in\mbz/l'p^{h-1}\mbz$ for all $i$. We have an obvious
map $pr:\tiThnh\ra\tiThnhq$ defined by reducing the diagonal entries
modulo $l'p^{h-1}\mbz$.

Let $\msKnhq$ be the free $\field$-module with basis $\{[A]_\vep\mid A\in\tiThnhq\}$.
We shall define an
algebra structure on $\msKnhq$ as follows. If the column sums of
$A$ are not equal to the row sums of $A'$ (as integers modulo
$l'p^{h-1}$), then the product $[A]_\vep\cdot[A']_\vep$ for $A,A'\in\tiThnhq$ is zero.
Assume now that the column sums of $A$ are equal to the row sums
of $A'$ (as integers modulo $l'p^{h-1}$). We can then represent $A,\, A'$
by elements $\ti A,\, \ti A'\in\tiThnh$ such that the column sums of
$\ti A$ are equal to the row sums of $\ti A'$ (as integers). According to
\ref{the property of K}(1), we can write $[\ti A]_\vep\cdot[\ti
A']_\vep=\sum_{\ti A''\in I}\rho_{{\ti A''}}[\ti A'']_\vep$ (product in
$\msKnh$) where $I=\{ \ti A''\in\tiThnh\ |\ \ro(\ti A'')=\ro(\ti A),\
\co(\ti A'')=\co(\ti A')\}$ (a finite set) and $\rho_{{\ti A''}}\in
\field$. Then the product $[A]_\vep\cdot[A']_\vep$ is defined to be $\sum_{\ti
A''\in I}\rho_{{\ti A''}}[pr(\ti A'')]_\vep$. From \ref{the property of K}(2) we see that the product is well defined and $\msKnhq$ becomes an associative
algebra over $\field$.

In the case where $l'$ is odd, the algebra $\msKnoq$ is the algebra $\msK'$ constructed in \cite[6.3]{BLM}. Furthermore, it was remarked at the end of \cite{BLM} that $\msK'$  is ``essentially'' the algebra defined in \cite[\S 5]{Lu90} for type $A$. We will prove in \ref{realization} that $\msKnhq$ is isomorphic to the algebra $\barUnkh$ in the case where $l'$ is odd.

Mimicking the construction of
$\hbfKn$,  we define $\hKnk$ to be the $\field$-module of
all formal $\field$-linear combinations
$\sum_{A\in\tiThn}\beta_A[A]_\vep$ satisfying the property \eqref{(F)}.
The product of two elements
$\sum_{A\in\tiThn}\beta_A[A]_\vep$,
$\sum_{B\in\tiThn}\gamma_B[B]_\vep$ in $\hKnk$ is defined to be
$\sum_{A,B}\beta_A\gamma_B[A]_\vep\cdot[B]_\vep$ where $[A]_\vep\cdot[B]_\vep$ is the
product in $\Knk$. Then $\hKnk$ becomes an associative algebra over $\field$.

We end this section by interpreting $\msKnhq$ as a $\field$-subalgebra of $\hKnk$.
For $h\geq 1$ let $\mbz_{l'p^{h-1}}=\mbz/l'p^{h-1}\mbz$ and let
$\bar\
:\mbz^n\ra(\mbz_{l'p^{h-1}})^n$  be the map defined by
$\ol{(j_1,j_2,\cdots,j_n)}=(\ol{j_1},\ol{j_2},\cdots,\ol{j_n}).$
For $A\in\Thnpmh$ and $\bar\mu\in(\mbz_{l'p^{h-1}})^n$ let
\begin{equation}\label{def of double bracket}
[\![A+\diag(\bar\mu)]\!]_h=\sum_{\nu\in\mbzn\atop\bar\mu
=\bar\nu}[A+\diag(\nu)]_\vep.
\end{equation}
Let $\Wnkh$ be the $\field$-submodule of $\hKnk$ spanned by the set $\{[\![A+\diag(\bar\la)]\!]_h\mid A\in\Thnpmh,\,\bar\la\in(\mbz_{l'p^{h-1}})^n\}$.
From \ref{the property of K} we see that $\Wnkh$ is a $\field$-subalgebra of $\hKnk$.
Furthermore, it is easy to prove that there is an algebra isomorphism
\begin{equation}\label{iso}
\Wnkh\tong\msKnhq
\end{equation}
defined by sending $[\![A]\!]_h$ to $[A]_\vep$ for $A\in\tiThnhq$.

\section{Realization of $\barUnkh$}
For $A\in\Thnpm$, $\dt\in\mbzn$ and $\la\in\mbnn$ let
$$A(\dt,\la)_\vep=\sum_{\mu\in\mbzn}\vep^{\mu\centerdot\dt}
\leb{\mu\atop\la}\rib_\vep[A+\diag(\mu)]_\vep\in\hKnk.$$
Let $\Vnk$ be the $\field$-submodule of $\hKnk$ spanned by the elements
$A(\dt,\la)_\vep$ for $A\in\Thnpm$, $\dt\in\mbzn$ and $\la\in\mbnn$.
For $h\geq 1$ let $\Vnkh$ be the $\field$-submodule of $\hKnk$ spanned by the elements
$A(\dt,\la)_\vep$ for $A\in\Thnpmh$, $\dt\in\mbzn$ and $\la\in\mbnnh$.
We will prove in \ref{realization} that $\barUnkh \cong  \Vnkh\cong\msKnhq$ in the case where $l'$ is odd, and that $\Unkh\cong\Vnkh\cong\msKnhq$ in the case where
$l'$ is even and $\field$ is a field.

Let $\hKn$ be the $\sZ$-submodule of $\hbfKn$ consisting of the elements
$\sum_{A\in\tiThn}\beta_A[A]$ with $\beta_A\in\sZ$. Then $\hKn$ is a $\sZ$-subalgebra of $\hbfKn$.
There is a natural algebra homomorphism $$\th:\hKn\ot_\sZ\field\ra\hKnk$$ defined by sending
$(\sum_{A\in\tiThn}\beta_A[A])\ot 1$ to $\sum_{A\in\tiThn}(\beta_A\cdot 1)[A]_\vep$, where $1$ is the identity element in $\field$.

Recall the injective algebra homomorphism $\vi:\bfUn\ra\hbfKn$ defined in \ref{BLM realization of bfUn}. From \ref{basis2 for Un} we see that
$\vi(\Un)\han\hKn.$ Thus, by restriction, we get a map
$\vi:\Un\ra\hKn$. It induces an algebra homomorphism
$\vi_\field:\Unk\ra\hKn\ot_\sZ \field$. The map $\th$, composed with $\vi_\field$ gives an algebra homomorphism
\begin{equation}\label{xi}
\xi:=\th\circ\vi_\field:\Unk\ra\hKnk.
\end{equation}

By definition we have
$\xi(A(\dt,\la))=A(\dt,\la)_\vep$
for $A\in\Thnpm$, $\dt\in\mbzn$ and $\la\in\mbnn$. This together with \ref{basis2 for Un} and \ref{span set of tiUnkh} implies that
\begin{equation}\label{xi Unk}
\xi(\Unk)=\Vnk \ \text{and}\ \xi(\Unkh)=\Vnkh.
\end{equation}
In particular, $\Vnk$ and $\Vnkh$ are all $\field$-subalgebras of $\hKnk$.

We will now construct several bases for $\Vnkh$ and $\Vnk$ in \ref{basis Vnkh odd} and \ref{basis Vnkh even}. These results will be used to prove
\ref{realization}.
According to \ref{Gauss identity1} we see that
$\bbl{\nu\atop\la}\bbr_\vep=\bbl{\nu+l'p^{h-1}\dt\atop\la}\bbr_\vep$ for $\la\in\mbnnh$ and $\nu,\dt\in\mbzn$. This implies that
\begin{equation}\label{A(dt,la)vep}
A(\dt,\la)_\vep=\sum_{\bar\mu\in(\mbz_{l'p^{h-1}})^n}\vep^{\dt\centerdot\mu}
\leb{\mu\atop\la}\rib_\vep[\![A+\diag(\bar\mu)]\!]_h
\end{equation}
for $A\in\Thnpmh$, $\dt\in\mbzn$ and $\la\in\mbnnh$, where
$[\![A+\diag(\bar\mu)]\!]_h$ is defined in \eqref{def of double bracket}.
For $\la,\mu\in\mbnn$, we write $\la\leq\mu$ if and only if $\la_i\leq\mu_i$ for $1\leq i\leq n$. If $\la\leq\mu$ and $\la_i<\mu_i$ for some $1\leq i\leq n$ then we write $\la<\mu$.

\begin{Lem}\label{basis Vnkh odd}
Assume $l'$ is odd. Then $\Vnkh=\Wnkh$ and the set
$\sN_h:=\{A(\bfl,\la)_\vep\mid A\in\Thnpmh,\,\la\in\mbnnh\}$
forms a $\field$-basis for $\Vnkh$.
Furthermore, if $p>0$, then the set $\sN:=\{A(\bfl,\la)\mid A\in\Thnpm,\,\la\in\mbnn\}$
forms a $\field$-basis for $\Vnk$.
\end{Lem}
\begin{proof}
From \eqref{A(dt,la)vep} we see that
for $A\in\Thnpmh$ and $\la\in\mbnnh$,
\begin{equation*}
A(\bfl,\la)_\vep=[\![A+\diag(\bar\la)]\!]_h+\sum_{\mu\in{\mbnnh},\,
\la<\mu}\leb{\mu\atop\la}\rib_\vep[\![A+\diag(\bar\mu)]\!]_h.
\end{equation*}
This, together with the fact that the set
$\sL_h$ forms a $\field$-basis for $\Wnkh$, shows that the set $\sN_h$ forms a $\field$-basis for $\Wnkh$.
It follows that $\Wnkh\han\Vnkh$. Furthermore from \eqref{A(dt,la)vep} we see that $\Vnkh\han \Wnkh.$ Thus $\Vnkh=\Wnkh$. Now we assume $p=\text{char}\field>0$. Since $\Vnk=\bin_{h\geq 1}\Vnkh$, $\sN=\bin_{h\geq 1}\sN_h$ and the set
$\sN_h$ forms a $\field$-basis for $\Vnkh$, we conclude that the set $\sN$ forms a $\field$-basis for $\Vnk$.
\end{proof}

\begin{Lem}\label{det}
For $m\geq 1$, let $X_m=((-1)^{\dt\centerdot\bt})_{\dt,\bt\in\sI_m}$, where $\sI_m=\{\dt\in\mbn^m\mid\dt_i\in\{0,1\} \,\text{for}\,1\leq i\leq m\}$. If we order the set $\sI_m$ lexicographically, then $\det(X_m)=(-2)^m$ for all $m$.
\end{Lem}
\begin{proof}
Since $\sI_m=\{(0,\dt)\mid\dt\in\sI_{m-1}\}\cup\{(1,\dt)\mid\dt\in\sI_{m-1}\}$
we see that
$$X_m=\begin{pmatrix} X_{m-1}& X_{m-1}\\ X_{m-1} &-X_{m-1}\end{pmatrix}.$$
This, together with the fact that $\det(X_1)=-2$, implies that
$$\det(X_m)=\det\begin{pmatrix} X_{m-1}& X_{m-1}\\ 0 &-2X_{m-1}\end{pmatrix}=-2\det(X_{m-1})^2=(-2)^{2^m-1}$$
as required.
\end{proof}
\begin{Coro}\label{basis Vnkh even}
Assume $l'$ is even and $\field$ is a field. Then $\Vnkh=\Wnkh$ and
the set $\sB_h:=\{A(\dt,\la)_\vep\mid A\in\Thnpmh,\,\la\in\mbnnh,\,\dt\in\mbnn,\,\dt_i\in\{0,1\},\,\forall i\}$
forms a $\field$-basis for $\Vnkh$.
Furthermore, if $p>0$, then the set $\sB:=\{A(\dt,\la)\mid A\in\Thnpm,\,\la,\dt\in\mbnn,\,\dt_i\in\{0,1\},\,\forall i\}$
forms a $\field$-basis for $\Vnk$.
\end{Coro}
\begin{proof}
Note that there is a bijective map from $\{(\dt,\la)\mid\dt\in\mbnn,\,\dt_i\in\{0,1\},\,\la\in\mbnnh\}$ to $(\mbz_{l'p^{h-1}})^n$ defined by sending $(\dt,\la)$ to $\ol{\la+lp^{h-1}\dt}$. Thus by
\eqref{A(dt,la)vep} and \ref{Gauss identity1}
we conclude that for $A\in\Thnpmh$, $\la\in\mbnnh$ and $\dt\in\mbnn$
\begin{equation*}
\begin{split}
A(\dt,\la)_\vep&=\sum_{\bt\in\mbnn,\,\bt_i\in\{0,1\},\,\forall
i\atop \al\in\mbnnh}\vep^{\dt\centerdot(\al+lp^{h-1}\bt)}
\leb\al+lp^{h-1}\bt\atop\la\rib_\vep[\![A+\diag(\ol{\al+lp^{h-1}\bt})]\!]_h\\
&=\sum_{\bt\in\mbnn,\,\bt_i\in\{0,1\},\,\forall
i\atop \al\in\mbnnh}
\vep^{\dt\centerdot\al}\vep^{lp^{h-1}(\dt\centerdot\bt-\bt\centerdot\la)}
\leb\al\atop\la\rib_\vep[\![A+\diag(\ol{\al+lp^{h-1}\bt})]\!]_h.
\end{split}
\end{equation*}
Since $l'$ is even and $(l',p)=1$ we see that $p$ is an odd prime.
This, together with the fact that $\vep^l=-1$, implies that $\vep^{lp^{h-1}}=(-1)^{p^{h-1}}=-1$. Thus for $A\in\Thnpmh$, $\la\in\mbnnh$ and $\dt\in\mbnn$
we have
\begin{equation}\label{even}
\begin{split}
A(\dt,\la)_\vep&=
\sum_{\bt\in\mbnn,\,\bt_i\in\{0,1\},\,\forall
i\atop \al\in\mbnnh}
\vep^{\dt\centerdot\al}(-1)^{\bt\centerdot(\dt-\la)}
\leb\al\atop\la\rib_\vep[\![A+\diag(\ol{\al+lp^{h-1}\bt})]\!]_h\\
&=\sum_{\bt\in\mbnn,\,\bt_i\in\{0,1\},\,\forall
i}\vep^{\dt\centerdot\la}(-1)^{\bt\centerdot(\dt-\la)}
[\![A+\diag(\ol{\la+lp^{h-1}\bt})]\!]_h\\
&\qquad\qquad+
\sum_{\bt\in\mbnn,\,\bt_i\in\{0,1\},\,\forall
i\atop \al\in\mbnnh,\,\la<\al}
\vep^{\dt\centerdot\al}(-1)^{\bt\centerdot(\dt-\la)}
\leb\al\atop\la\rib_\vep[\![A+\diag(\ol{\al+lp^{h-1}\bt})]\!]_h.
\end{split}
\end{equation}
From \ref{det} we see that for $\la\in\mbnnh$,
$$\det(\vep^{\dt\centerdot\la}
(-1)^{\bt\centerdot(\dt-\la)})_{\dt,\bt\in\sI_n}
=(-\vep)^{\sum_{\dt\in\sI_n}\la\centerdot\dt}(-2)^{2^n-1}=
(-\vep)^{\sum_{\dt\in\sI_n}\la\centerdot\dt}(\vep^l-1)^{2^n-1}\not=0,$$
where $\sI_n=\{\dt\in\mbn^n\mid\dt_i\in\{0,1\} \,\text{for}\,1\leq i\leq n\}$. It follows that the martix
$(\vep^{\dt\centerdot\la}
(-1)^{\bt\centerdot(\dt-\la)})_{\dt,\bt\in\sI_n}$ is invertible
since $\field$ is a field. Thus by \eqref{even} we conclude that
the set $\sB_h$ forms a $\field$-basis for $\Wnkh$ and
$\Vnkh=\Wnkh$.  Now we assume $p=\text{char}\field>0$. Then $\sB=\bin_{h\geq 1}\sB_h$. Since the set
$\sB_h$ is linear independent for all $h$, we conclude that the set $\sB$ is linear independent. Consequently, the set $\sB$ forms a $\field$-basis for $\Vnk$.
\end{proof}

We are now ready to  prove the main result of this paper.

\begin{Thm}\label{Ker(xi)}
$(1)$ If $l'$ is odd, then $\ker(\xi)=\lan K_i^l-1\mid 1\leq i\leq n\ran$ and hence $\Unk/\lan K_i^l-1\mid 1\leq i\leq n\ran\cong\Vnk$.

$(2)$ If $l'$ is even and $\field$ is a field with $p=\mathrm{char}\field>0$, then $\xi$ is injective and hence $\Unk\cong\Vnk$.
\end{Thm}
\begin{proof}
The assertion (1) can be proved in a way similar to the proof of \cite[4.6]{Fu12}. The assertion (2) follows from \ref{basis2 for Un}, \ref{basis Vnkh even} and \eqref{xi Unk}.
\end{proof}

\begin{Thm}\label{realization}
$(1)$ If $l'$ is odd, then $\barUnkh \cong  \Vnkh\cong\msKnhq$ for $h\geq 1$.

$(2)$ If $l'$ is even and $\field$ is a field, then $\Unkh\cong\Vnkh\cong\msKnhq$ for $h\geq 1$.
\end{Thm}
\begin{proof}
If either $l'$ is odd or both $l'$ is even and $\field$ is a field, then by \eqref{iso}, \ref{basis Vnkh odd} and \ref {basis Vnkh even}, we deduce that $\Vnkh\cong\msKnhq$.
If $l'$ is odd, then $\xi(K_i^l-1)=0$ and hence the map $\xi:\Unk\ra\hKnk$ induces an algebra homomorphism
$$\bar\xi:\Unk/\lan K_i^l-1\mid 1\leq i\leq n\ran\ra\hKnk.$$
One can prove that
the set $\{E^{(A^+)}\prod_{1\leq i\leq n}K_i^{-\la_i}\bbl{K_i;0\atop\la_i}\bbr F^{(A^-)}\mid A\in\Thnpmh,\,\la\in\mbnnh\}$ forms a $\field$-basis of $\barUnkh$ in a way similar to the proof of \cite[6.5]{Lu90}. Thus we may regard $\barUnkh$ as a $\field$-subalgebra of $\Unk/\lan K_i^l-1\mid 1\leq i\leq n\ran$. From \eqref{xi Unk} we see that $\bar\xi(\barUnkh)=\Vnkh$. Thus the restriction of $\bar\xi$ to $\barUnkh$ yields a surjective algebra homomorphism
$$\bar\xi':\barUnkh \twoheadrightarrow\Vnkh.$$
This, together with \ref{Ker(xi)}(1), implies that $\barUnkh\cong\Vnkh$. Now we assume
$l'$ is even and $\field$ is a field. Since $\xi(\Unkh)=\Vnkh$ by \eqref{xi Unk}, the restriction of $\xi$ to $\Unkh$ yields a
surjective algebra homomorphism
$$\xi':\Unkh \twoheadrightarrow\Vnkh.$$
From  \ref{basis of Unkh} and \ref{basis Vnkh even} we see that $\xi'$ is injective. Consequently, $\Unkh\cong\Vnkh$.
\end{proof}

\section{The infinitesimal $q$-Schur algebras and little $q$-Schur algebras}

Let $\Sr$ be the algebra over $\sZ$ introduced in
\cite[1.2]{BLM}. It has a $\sZ$-basis
$\{[A]\mid A\in\Thnr\}$ defined in \cite{BLM}, where
$\Thnr=\{A\in\Thn\mid\sg(A):=\sum_{1\leq i,j\leq n}a_{i,j}=r\}$.
It is proved in \cite[A.1]{Du92} that
the algebra $\Sr$ is isomorphic to the $q$-Schur algebra introduced in \cite{DJ89,DJ91}.
Let $\Srk=\Sr\ot_\sZ\field$. For $A\in\Thnr$ let
$$[A]_\vep=[A]\ot 1\in\Srk.$$

Let $\Lanr=\{\la\in\mbnn\mid\sum_{1\leq i\leq n}\la_i=r\}$ and $\ol{\La(n,r)}_{l'p^{h-1}}=\{\bar\la\in(\mbz_{l'p^{h-1}})^n\mid\la\in
\Lanr\}$.
For $A\in\Thnpmh$ and $\lb\in(\mbz_{l'p^{h-1}})^n$
we define the element $[\![A+\diag(\lb),r]\!]_h\in\Srk$ as follows:
\begin{equation*}\label{[[A,r]]}
[\![A+\diag(\lb),r]\!]_h=
\begin{cases}
\sum\limits_{\mu\in\La(n,r-\sg(A)) \atop
\mb=\lb}[A+\diag(\mu)]_\vep &\text{if $\sg(A)\leq r$ and $\lb\in\ol{\La(n,r-\sg(A))}_{l'p^{h-1}}$,}\\
0&\text{otherwise}.
\end{cases}
\end{equation*}
Let $\Unkhr$ be the $\field$-submodule of $\Srk$ spanned by the set
$\{[\![A+\diag(\lb),r]\!]_h\mid A\in\Thnpmh,\,\lb\in(\mbz_{l'p^{h-1}})^n\}$.
According to \cite[4.8]{DFW12}, $\Unkhr$ is a $\field$-subalgebra of
$\Srk$. Note that the algebra $\Unkor$ is the little $q$-Schur algebra introduced in \cite{DFW05,Fu07}. We will prove in \ref{little q-Schur algebras} that the algebra $\Unkhr$ is a homomorphic image of $\Unkh$.

Let $\bfSr=\Sr\ot_\sZ\mbq(v)$.
For $A\in\Thnpm$, $\dt\in\mbzn$
let
\begin{equation*}
\begin{split}
A(\dt,r)&=\sum_{\mu\in\La(n,r-\sg(A))}v^{\mu\centerdot\dt}
[A+\diag(\mu)]\in\bfSr.
\end{split}
\end{equation*}
According to \cite{BLM}, there is
an algebra epimorphism $$\zeta_r:\bfUn
\twoheadrightarrow\bfSr$$ satisfying
$\zeta_r(E_i)=E_{i,i+1}(\mathbf 0,r)$, $\zeta_r(K_1^{j_1}K_2^{j_2}\cdots
K_n^{j_n})=0(\mathbf j,r)$ and $\zeta_r(F_i)=E_{i+1,i}(\mathbf
0,r)$, for $1\leq i\leq n-1$ and $\bfj\in\mbzn$.
It is proved in \cite{Du95} that $\zeta_r(\Un)=\Sr$.
By restriction, the map $\zeta_{r}:\bfUn\ra\bfSr$ induces
a surjective algebra homomorphism $\zeta_r:\Un\ra\Sr$.
The map $\zeta_r:\Un\ra\Sr$ induces
an algebra homomorphism
$$\zeta_{r,\field}:=\zeta_r\ot id:\Unk\ra\Srk.$$
\begin{Prop}\label{little q-Schur algebras}
If either $l'$ is odd or both $l'$ is even and $\field$ is a field then $\zrk(\Unkh)=\Unkhr$.
\end{Prop}
\begin{proof}
According to \cite[6.7]{DF09}, there is a surjective algebra homomorphism
$$\dzr:\Kn\ra\Sr$$
such that $$\dzr([A])=\begin{cases}[A]& \mathrm{if\ }A\in\Thnr;\\
0&  \mathrm{otherwise}.\end{cases}$$
The map $\dzr$ induces a surjective algebra homomorphism
$$\hdzrk:\hKnk\ra\Srk$$ defined by sending
$\sum_{A\in\tiThn}\bt_A[A]_\vep$ to $\sum_{A\in\Thnr}\bt_A[A]_\vep$.
It is easy to see that
\begin{equation}\label{hdzrk circ xi}
\zrk=\hdzrk\circ\xi
\end{equation}
where $\xi$ is given in \eqref{xi}. This together with
\eqref{xi Unk} implies that
$\zrk(\Unkh)=\hdzrk(\Vnkh)$. Clearly, for $A\in\Thnpmh$ and $\bar\la\in(\mbz_{l'p^{h-1}})^n$, we have
$\hdzrk([\![A+\diag(\bar\la)]\!]_h)=[\![A+\diag(\bar\la),r]\!]_h$.
Combining these facts with \ref{basis Vnkh odd} and \ref{basis Vnkh even} gives the result.
\end{proof}

Let $\snkhr$ be the the infinitesimal $q$-Schur algebra introduced in \cite{Cox,Cox00}. The algebra $\snkhr$ is a certain $\field$-subalgebra of the $q$-Schur algebra $\Srk$.
According to \cite[5.3.1]{Cox}, we have the following result.
\begin{Lem}\label{basis of infinitesimal q-Schur algebra}
The set
$\{[A]_\vep\mid A\in\Thnrh\}$ forms a $\field$-basis of $\snkhr$.
\end{Lem}

For $h\geq 1$ let $\snkh$ be the $\field$-subalgebra of $\Unk$ generated by the elements $E_{i}^{(m)}$, $F_i^{(m)}$, $K_j^{\pm 1}$, $\bbl{K_j;0\atop t}\bbr$ for $1\leq i\leq n-1$, $1\leq j\leq n$, $t\in\mbn$ and $0\leq m<lp^{h-1}$. We will prove in \ref{infinitesimal q-Schur algebras} that the algebra $\snkhr$ is a homomorphic image of $\snkh$.

\begin{Lem}\label{snkh}
Each of the following set forms a $\field$-basis for $\snkh:$
\begin{itemize}
\item[(1)]
$\{E^{(A^+)}\prod_{1\leq i\leq n}K_i^{\dt_i}\bbl{K_i;0\atop\la_i}\bbr F^{(A^-)}\mid A\in\Thnpmh,\,\dt\in\mbnn,\,\dt_i\in\{0,1\},\,\forall i,\,\la\in\mbnn\};$
\item[(2)]
$\{A(\dt,\la)\mid A\in\Thnpmh,\,\dt,\la\in\mbnn,\,\dt_i\in\{0,1\},\,\forall i\}.$
\end{itemize}
\end{Lem}
\begin{proof}
The assertion can be proved in a way similar to the proof of \ref{basis of Unkh}.
\end{proof}

\begin{Prop}\label{infinitesimal q-Schur algebras}
We have $\zeta_{r,\field}(\snkh)=\snkhr$.
\end{Prop}
\begin{proof}
From \ref{hdzrk circ xi}
we see that
$$\zrk(A(\dt,\la))=\hdzrk(A(\dt,\la)_\vep)= A(\dt,\la,r)_\vep$$
for all $A,\dt,\la$, where
$A(\dt,\la,r)_\vep=\sum_{\mu\in\La(n,r-\sg(A))}\vep^{\mu\centerdot\dt}
\leb{\mu\atop\la}\rib_\vep[A+\diag(\mu)]_\vep\in\Srk$. Thus by \ref{basis of infinitesimal q-Schur algebra} and \ref{infinitesimal q-Schur algebras} we conclude that
$$\zrk(\snkh)=\spann_\field\{A(\dt,\la,r)_\vep\mid A\in\Thnpmh,\,\dt,\la\in\mbnn,\,\dt_i\in\{0,1\},\,\forall i\}\han\snkhr.$$
On the other hand, for $A\in\Thnpmh$ and $\mu\in\La(n,r-\sg(A))$ we have
$[A+\diag(\mu)]=A(\bfl,\mu,r)\in\zrk(\snkh)$. This implies that
$\snkhr\han\zrk(\snkh)$. The assertion follows.
\end{proof}

\end{document}